\documentclass[10pt]{amsart}
\usepackage{amsmath,amssymb,enumerate}

\newtheorem{thm}{Theorem}[section]
 \numberwithin{equation}{section} %% Comment out for sequentially-numbered
 \numberwithin{figure}{section} %% Comment out for sequentially-numbered
 \theoremstyle{plain}
 \theoremstyle{plain}    
 \theoremstyle{plain}    
 \newtheorem{prop}[thm]{Proposition} %%Delete [thm] to re-start numbering
 \theoremstyle{plain}    
 \newtheorem{lem}[thm]{Lemma} %%Delete [thm] to re-start numbering
 \newtheorem{rem}[thm]{Remark}
  \newtheorem{defi}[thm]{Definition}
\newtheorem{exa}[thm]{Example}

\usepackage{hyperref}
\usepackage{xcolor}
\usepackage{comment}
\usepackage{dsfont}
\usepackage{yhmath}
\definecolor{violet}{rgb}{0.0,0.2,0.7}
\definecolor{rouge}{cmyk}{0.0,0.6,0.4,0.3}
\definecolor{rouge2}{rgb}{0.8,0.0,0.2}
\hypersetup{
    bookmarks=true,         % show bookmarks bar?
    unicode=false,          % non-Latin characters in Acrobatâs bookmarks
    pdftoolbar=true,        % show Acrobatâs toolbar?
    pdfmenubar=true,        % show Acrobatâs menu?
    pdffitwindow=false,     % window fit to page when opened
    pdfstartview={FitH},    % fits the width of the page to the window
    pdftitle={},    % title
    pdfauthor={},     % author
    colorlinks=true,       % false: boxed links; true: colored links
   linkcolor=blue,          % color of internal links
    citecolor=blue,        % color of links to bibliography
    filecolor=black,      % color of file links
    urlcolor=cyan}           % color of external links

\newcommand{\R}{\mathbb{R}}
\newcommand{\C}{\mathbb{C}}

\newcommand{\Dc}{\mathcal{D}}

\newcommand{\f}{\varphi}
\newcommand{\p}{\psi}
\newcommand{\Ec}{\mathcal{E}}

\newcommand{\SH}{\mathcal{SH}}
\newcommand{\ind}{{\bf 1}}

\newcommand{\setdef}{\ \big\vert \ }
\newcommand{\Capa}{{\rm Cap}}

\newcommand{\SHXo}{\SH_m(X,\omega)}
\newcommand{\vep}{\varepsilon}

\newcommand{\EcXo}{\mathcal{E}(X,\omega,m)}
\newcommand{\Ect}{\mathcal{E}^1(X,\omega,m)}

\newcommand{\Capm}{\Capa_{\omega,m}}

\newtheorem{theorem}{Theorem}

\begin{document}
\title{Mixed Hessian inequalities and uniqueness in the class $\EcXo$} 
\date{\today \\
The first-named author is partially supported by NCN grant 2013/08/A/ST1/00312.\\
The second-named author is  partially supported by the french ANR project MACK} 

\author{S{\l}awomir Dinew} 

\address{Faculty of Mathematics and Computer Science\\ Jagiellonian University 30-348 Krakow\\ Lojasiewicza 6, Poland}

\email{slawomir.dinew@im.uj.edu.pl}

\author{Chinh H. Lu}
\address{Chalmers University of Technology \\ Mathematical Sciences\\
412 96 Gothenburg\\ Sweden}

\email{chinh@chalmers.se}

\begin{abstract}
We prove a general inequality for mixed Hessian measures by global arguments. Our method also yields a simplification for the case of  complex Monge-Amp\`ere equation. Exploiting this and using 
Ko{\l}odziej's mass concentration technique we also prove the uniqueness of the solutions to the complex Hessian equation on compact K\"ahler manifolds in the case of probability measures vanishing on $m$-polar sets.
\end{abstract}

\maketitle

\section{Introduction}
Nonlinear equations of Monge-Amp\`ere and more generally of Hessian type have proven to be a very fruitful branch of research. In the set-up of compact K\"ahler manifolds the solution of the Calabi conjecture by Yau (\cite{Y78}) has literally opened the door for PDE methods in complex geometry. In fact Yau's theorem is still a subject of numerous generalizations (see, in particular \cite{Kol98, BBGZ13} and references therein) and new powerful tools coming from pluripotential theory allowed applications that were previously unreachable.

The complex Hessian equation on compact K\"ahler manifolds is not that geometric because the solutions do not yield K\"ahler metrics (see hovewer \cite{AV10} for some geometric applications). Nevertheless the PDE theory is interesting on its own right, and a strong motivation for considering it is its real counterpart that has been developed some time ago thanks to the works of Trudinger, Wang, Chou and others (see \cite{CNS85, CW01, Tr95, TW99,W09} and references therein). Interestingly while many estimates known from the Monge-Amp\`ere theory apply verbatim, some like the a priori gradient estimate for the Hessian equation turned out to be unexpectedly difficult (this estimate was finally proven in \cite{DK12}). This finally terminated the program for proving an analogue of the Calabi-Yau theorem in the Hessian setting. Afterwards a fruitful potential theory was established culminating in the solution of the smoothing problem for m-subharmonic functions (see \cite{Pl13, LN14}).

In this note we continue the investigation of the pluripotential theory for complex Hessian equations on compact K\"ahler manifolds. We establish some technical results that allow us to prove a general inequality for mixed Hessian measures (an analogue of the main result in \cite{Dinew07}). It is worth mentioning that while the methods from \cite{Dinew07} are local in spirit, here our approach is global which allows an essential simplification. In particular we avoid the delicate approximation by Dirichlet solutions with potentially discontinuous boundary values from \cite{Dinew07}. Note that the local inequalities can be deduced from the global ones (see Theorem \ref{thm: global to local}) thus our method yields a simplification even for the Monge-Amp\`ere equation. 

\begin{theorem}\label{theorem 1}[Theorem \ref{thm: mixed hessian inequality}]
Let $u_1,...,u_m$ be functions in $\EcXo$ and $\mu$ be a positive Radon measure vanishing on $m$-polar sets such that 
$$
H_m(u_j) \geq f_j \mu \ , \forall j=1,...,m,
$$
where $f_j$ are non-negative integrable (with respect to $\mu$) functions. Then  one has
$$
\omega_{u_1} \wedge \cdots \wedge \omega_{u_m}\wedge \omega^{n-m} \geq (f_1\cdots f_m)^{1/m} \mu .
$$
\end{theorem}

Having these inequalities in hand it is straightforward to generalize the uniqueness theorem from \cite{Dinew09} thus establishing a very general uniqueness result for the solutions to the complex Hessian equation living in the class 
$\EcXo$.  

\begin{theorem}\label{theorem 2}[Theorem \ref{thm: uniqueness}]
Let $u,v$ be functions in $\EcXo$  such that 
$$H_m(u)=H_m(v).$$
Then $u-v$ is a constant.
\end{theorem}

The note is organized as follows. First we briefly recall the notions and tools that we shall use later on. Then in Section \ref{Ineq} we establish the inequality for mixed Hessian measures. In Section \ref{Uniq} we prove the uniqueness theorem. Finally in Section \ref{example} we give an explicit example that the mixed Hessian inequality fails outside the class $\EcXo$.

\section{Preliminaries}
We recall basic facts on $m$-subharmonic functions which will be used  later. We first  consider the local setting. 
\subsection{$m$-subharmonic functions}
Let $M$ be a non-compact K\"ahler manifold of dimension $n$ and  $\omega$ be a K\"ahler form. Fix an integer $m$ such that $1\leq m\leq n$. 
\begin{defi}\label{def: m-sh smooth}
A smooth function $u$ is called $m$-subharmonic ($m$-sh for short) on $M$ if the following holds in the classical sense
$$
(dd^c u)^k \wedge \omega^{n-k} \geq 0 , \ \forall k=1, \cdots, m .
$$ 
Equivalently, $u$ is $m$-sh if the vector of eigenvalues $\lambda(x) \in \R^n$ of $dd^c u$ with respect to $\omega$ satisfies
$$
S_k(\lambda(x)) \geq 0 , \ \forall x\in M,\ \forall k=1, \cdots, m .
$$ 
Yet another characterization of $m$-subharmonicty can be obtained via the inequalities 
$$
\forall \vep>0\ \ (dd^c u+ \vep \omega)^m \wedge \omega^{n-m} \geq 0. 
$$ 
\end{defi}

From G{\aa}rding's inequality \cite{Ga59} we easily get the smooth version of the mixed Hessian inequalities:
\begin{lem}
\label{lem: Garding}
Let $u_1,...,u_m$ be smooth $m$-sh functions on $M$ and $f_1,...,f_m$ be smooth nonnegative functions such that 
$$
(dd^c u_j)^m \wedge \omega^{n-m} =f_j\omega^n, \ \forall j=1,...,m.
$$
Then the following mixed Hessian inequality holds
$$
dd^c u_1 \wedge ...\wedge dd^c u_m \wedge \omega^{n-m}
\geq (f_1...f_m)^{1/m}\omega^n .
$$
\end{lem}

\begin{defi}\label{def: m-sh non-smooth}
Let $u$ be a locally integrable, upper semicontinuous function on $M$.
Then $u$ is called $m$-sh on $M$ if the following two conditions are satisfied:
\begin{enumerate}[i)] 
\item in the weak sense of currents  $$dd^c u\wedge dd^c \f_1\wedge \cdots \wedge dd^c \f_{m-1} \wedge \omega^{n-m}\geq 0, $$
for any collection  $\f_1\cdots\f_{m-1}$ of smooth $m$-sh functions;
\item  if $v$ is another function satisfying the above inequalities and $v=u$ almost everywhere on $M$ then $u\leq v$. 
\end{enumerate}
\end{defi}

Note that two $m$-subharmonic functions are the same if they are equal almost everywhere on $M$. Observe also that by G{\aa}rding's inequality Definition \ref{def: m-sh smooth} and Definition \ref{def: m-sh non-smooth} are equivalent for smooth functions. The class of $m$-sh functions on $M$ (with respect to $\omega$) is denoted by 
$\SH_m(M)$. Note that this class depends heavily on the metric $\omega$ which makes the regularization process of $m$-sh functions very complicated. More precisely, it is not clear whether the convolution of a $m$-sh functions (when $\omega$ is not flat) with a smooth kernel is $m$-sh. Nevertheless any $m$-sh function can be approximated by smooth $m$-sh functions (see \cite{Pl13}, \cite{LN14}).

The smoothing property allows us to define the domain of definition of the complex Hessian operator:
\begin{defi}
A $m$-sh function $u$ belongs to the domain of definition of the Hessian operator $\Dc_m(M)$ if there exists a regular Borel measure
$\mu$ such that for any sequence $(u_j)$ of smooth $m$-sh functions decreasing to $u$ then $(dd^c u_j)^m \wedge \omega^{n-m}$ converges weakly to $\mu$. 
\end{defi}

We shall need the following result of B\l ocki \cite{Bl05}:
\begin{lem}\cite{Bl05}
\label{lem: W12 Blocki}
Assume that $M$ is an open subset of $\C^n$ and $\omega=dd^c |z|^2$ is the standard K\"ahler form in $\C^n$. If $m=2$ and $u$ is $m$-sh on $M$ then $u\in \Dc_m(M)$ if and only if $u\in W^{1,2}_{\rm loc}(M)$.
\end{lem}

\subsection{$\omega$-$m$-subharmonic functions}
Let $(X,\omega)$ be a compact K\"ahler manifold of dimension $n$ and $m$ be an integer such that $1\leq m\leq n$. 
\begin{defi}
A function $u: X\rightarrow \R\cup\{-\infty\}$ is called $\omega$-$m$-subharmonic ($\omega$-$m$-sh for short) if in any local chart $\Omega$ of $X$, the function 
$\rho + u$ is $m$-sh, where $\rho$ is a local potential of $\omega$. 
\end{defi}
Observe that a smooth function $u$ is $\omega$-$m$-sh if and only if 
$$
(\omega+dd^c u)^k \wedge \omega^{n-k} \geq 0 ,\ \forall k=1,...,m,
$$
or equivalently iff
$$
\ \forall 
\vep >0\ \ ((1+\vep)\omega+dd^c u)^m \wedge \omega^{n-m} \geq 0.
$$

We let $\SH_m(X,\omega)$ denote the class of $\omega$-$m$-sh functions 
on $X$. It follows from \cite{LN14} that for any $u\in \SHXo$ there exists a decreasing  sequence  of smooth $\omega$-$m$-sh functions on $X$ which converges to $u$. Following the classical pluripotential method of Bedford and Taylor \cite{BT76} one can then define the complex Hessian operator for any bounded $\omega$-$m$-sh function:
$$
H_m(u):=(\omega +dd^c u)^m\wedge \omega^{n-m},
$$
which is a non-negative (regular) Borel measure on $X$. The complex Hessian operator is local in the plurifine topology which in particular implies that the sequence 
$$
\ind_{\{u>-j\}} H_m(\max(u,-j)) 
$$
is non-decreasing, for any $u\in \SHXo$. Moreover given any $\omega$-$m$-sh function $u$
$$\forall j\in \mathbb N,\ \int_X\ind_{\{u>-j\}} H_m(\max(u,-j))\leq\int_X\omega^n.$$

We then define the class $\EcXo$  as the set of these of $\omega$-$m$-sh functions for  which 
$$
\lim_{j\rightarrow\infty}\int_X\ind_{\{u>-j\}} H_m(\max(u,-j))=\int_X\omega^n,
$$
(see \cite{LN14}, \cite{GZ07}). For any $u\in\EcXo$ we define 
$$H_m(u):= \lim_{j\rightarrow\infty}\ind_{\{u>-j\}} H_m(\max(u,-j)).$$
 We recall the following result:
\begin{thm}\cite{LN14}
Let $\mu$ be a non-negative Radon measure on $X$ vanishing on all $m$-polar sets. Assume that $\mu(X)=\int_X \omega^n$. Then there exists 
$u\in \EcXo$ such that 
$$
H_m(u)=\mu.
$$
\end{thm}

Finally by $\Ect$ we denote the subset of $\EcXo$ of all $u$ integrable with respect to their own Hessian measure $H_m(u)$.

Note that when $\mu=f\omega^n$ for some smooth positive function $f$ then $u$ is also smooth (see \cite{DK12}). The variational method used in \cite{LN14} (which originated from \cite{BBGZ13}) can be applied in the same way to get the following result:

\begin{thm}\label{thm: exponential}
Let $\mu$ be a non-negative Radon measure on $X$ vanishing on all $m$-polar sets. Then for any $\vep>0$ there exists $u\in \Ec^1(X,\omega,m)$ such that
$$
H_m(u)=e^{\vep u}\mu.
$$
\end{thm}
The proof of Theorem \ref{thm: exponential} is left to the reader as an easy exercise. We will need this result in the next section. We remark that the solution $u$ is unique- this follows from the domination principle in $\EcXo$ which is an easy consequence of Theorem \ref{theorem 2} as shown in \cite{BL12} (when $m=n$). 

Next we shall need the notion of capacity associated to the Hessian operator and convergence with respect to it:
\begin{defi}
The $m$-capacity of a Borel subset of $X$ is defined as 
$$
\Capm(E) : =\sup\left\{ \int_E H_m(u) \setdef u\in \SHXo , \ -1\leq u\leq 0\right\}.
$$
\end{defi}
\begin{defi}
A sequence $(u_j)$ of  converges in $\Capm$ to $u$ if for any $\vep>0$
$$
\lim_{j\to +\infty} \Capm(|u_j-u|>\vep) =0.
$$
\end{defi}
Exactly as it the plurisubharmonic setting we have convergence in capacity for monotonely decreasing sequences and quasi-continuity of $\omega$-$m$-sh functions:
\begin{prop}
If $(u_j)\subset \SHXo$ decreases to $u \not \equiv -\infty$ then 
$u_j$ converges to $u$ in $\Capm$.
\end{prop}

\begin{lem}
Any $\omega$-$m$-sh function $u$ is quasi-continuous, i.e. for any $\vep>0$ there exists an open subset $U$ such that $\Capm(U)<\vep$ and 
$u$ restricted on $X\setminus U$ is continuous. 
\end{lem}
The following convergence result is an easy adaptation of \cite{Kol05}, Corollary 1.14:

\begin{lem}
\label{lem: cap convergence}
Let $(\varphi^1_j),...,(\varphi^m_j)$ be uniformly bounded sequence of functions
in $\SHXo\cap L^{\infty}(X)$ converging in $\Capm$ to $\varphi^1,...,\varphi^m$ respectively. Assume that $(f_j)$ is a uniformly bounded sequence of quasi continuous functions converging in 
$\Capm$ to $f$. Then we have the weak convergence of measures  
\begin{equation*}\label{prop: convergence in capacity 1}
f_j\omega_{\varphi_j^1}\wedge...\wedge\omega_{\varphi_j^m}\wedge
\omega^{n-m}\rightharpoonup
f\omega_{\varphi^1}\wedge...\wedge\omega_{\varphi^m}\wedge \omega^{n-m}.
\end{equation*}
\end{lem}

\section{An inequality for mixed Hessian measures}\label{Ineq}
Our main goal in this section is to prove an inequality for mixed Hessian measures (Theorem \ref{theorem 1}). We first prove some technical results which we need later on. Our approach, which makes use of the "$\beta$-convergence" method of Berman \cite{Berman13}, is global in nature and therefore avoids some difficulties arising from rough boundary data present in the local Dirichlet problem (compare with \cite{Dinew07}). 

Let $\mu$ be a positive Radon measure on $X$. Let $L_{\mu}$ denote the mapping 
$$
\SHXo \ni \f  \mapsto L_{\mu}(\f) := \int_X \f d\mu .
$$
The following result has been proven in \cite{LN14} (it is in fact an an easy adaptation of a theorem from \cite{BBGZ13}):
\begin{lem}
\label{lem: continuity}
Let $\mu$ be a positive Radon measure on $X$ which is dominated by $\Capm$. Then $L_{\mu}$ is continuous on $\Ect$ with respect to the $L^1$ topology.
\end{lem}

\begin{lem}
\label{lem: convergence}
Let $(u_j)$ and $(\f_j)$ be  sequences of bounded $\omega$-$m$-sh functions on $X$.
Assume that $u_j$ converges in $\Capm$ to $u\in \SHXo \cap L^{\infty}(X)$ and  $\f_j$ converges in $L^1(X)$ to $\f\in \SHXo \cap L^{\infty}(X)$. If there exists $A>0$ such that 
$$
H_m(\f_j) \leq A\, H_m(u_j) ,\ \forall j ,
$$
then $\f_j \rightarrow \f$  in $\Capm$. In particular, $H_m(\f_j)$ converges weakly to $H_m(\f)$.
\end{lem}

\begin{proof}
Without loss of generality we can assume that all the functions are negative. Fix $\vep>0$ and set $E_j:=\{\f<\f_j-3\vep\}$. We are going to prove that $\Capm(E_j)\to 0$ as $j\to+\infty$. Let $C>1$ be a  constant such that 
$$
\sup_X (|u_j|+|\f_j|+|u|+|\f|) \leq C\ , \ \forall j .
$$
Fix $1>\delta>0$ such that $\delta C<\vep$ and let $v\in \SHXo$ such that $-1\leq v\leq 0$. The following inclusions are obvious:
$$
\{\f<\f_j-3\vep\} \subset  \{\f<(1-\delta)\f_j +\delta v-2\vep\}\subset \{\f<\f_j -\vep\}.
$$
By the comparison principle we thus get 
$$
\delta^m \int_{\{\f<\f_j-3\vep\}} H_m(v) \leq \int_{\{\f<\f_j-\vep\}} H_m(\f)\leq \vep^{-1}\int_X [\max(\f,\f_j)-\f ] H_m(\f).
$$
Taking the supremum over all $v$ and using Lemma \ref{lem: continuity} we obtain 
$$
\lim_{j\to+\infty} \Capm(E_j) = 0 .
$$
Now, set $F_{\vep}:=\{\f_j<\f-3\vep\}$. It remains to prove that $\Capm(F_j)$ also converges to $0$ as $j\to+\infty$. Arguing as in the first step we get
\begin{eqnarray*}
\delta^m \Capm(F_j) &\leq & \vep^{-1}\int_X \left[\max(\f_j,\f)-\f_j\right]H_m(\f_j)\\
&\leq & A\vep^{-1}\int_X \left[\max(\f_j,\f)-\f_j\right]H_m(u_j).
\end{eqnarray*}
On the other hand, since 
$$
\int_X \left[\max(\f_j,\f)-\f_j\right]H_m(u) \longrightarrow 0
$$
as follows from the first step, it suffices to prove that 
$$
\int_X \left[\max(\f_j,\f)-\f_j\right] \left[ H_m(u_j)-H_m(u)\right] \longrightarrow 0 .
$$
By an integration by parts the above term is dominated by
$$
C_1 \int_X |u_j-u| H_m(\p_j),
$$
where $\p_j$ is a sequence of uniformly bounded $\omega$-$m$-sh functions and $C_1$ depends only on $C$. Now, fix a small $t>0$. Since 
$u_j$ converges to $u$ in $\Capm$ we get for some $C_2$ also dependent only on $C$ that
\begin{eqnarray*}
\limsup_{j\to+\infty}\int_X |u_j-u| H_m(\p_j) &\leq & t \int_X \omega^n +  \limsup_{j\to+\infty} C_2 \Capm(\{|u_j-u|>t\} \\
&=&t\int_X \omega^n ,
\end{eqnarray*}
from which the result follows.
\end{proof}

The uniqueness result for bounded $\omega$-$m$-sh functions can be proven by repeating the arguments in \cite{Bl03}: 
\begin{lem}
\label{lem: uniqueness bounded}
Let $u,v$ be bounded $\omega$-$m$-sh functions. If $H_m(u)=H_m(v)$ then $u-v$ is constant.
\end{lem}
We shall also need the following domination principle:
\begin{lem}
\label{lem: domination principle}
Let $u,v$ be bounded $\omega$-$m$-sh functions. Assume that $H_m(u)$ vanishes on the set $\{u<v\}$. 
Then $u\geq v$ on $X$.
\end{lem}
\begin{proof}
We can assume that $v\leq 0$. Fix $\vep>0$ and $1>s>0$ such that $s\max(1,\sup_X |v|)<\vep$. Let
$\f$ be a $\omega$-$m$-sh function on $X$ such that $-1\leq \f\leq 0$.
Applying the comparison principle we obtain

\begin{eqnarray*}
s^m \int_{\{u<v-2\vep\}} H_m(\f) &\leq &\int_{\{u<v-2\vep\}} H_m((1-s) v+s\f)\\
&\leq & \int_{\{u<(1-s)v+s\f-\vep\}} H_m(s\f +(1-s) v)\\
&\leq & \int_{\{u<(1-s)v+s\f-\vep\}} H_m(u)\\
&\leq & \int_{\{u<v\}} H_m(u)=0.\\
\end{eqnarray*}
We then get $\Capm(\{u<v-2\vep\})=0$ which implies the result.
\end{proof}

From the domination principle we immediately get the following:
\begin{lem}
\label{lem: comparison}
Let $u, v$ be bounded $\omega$-$m$-sh functions. Assume that $\mu$ is a positive Radon measure such that 
$$
H_m(u) \leq e^u \mu \ {\rm and}\ H_m(v) \geq e^v \mu .
$$
Then $v\leq u$. 
\end{lem}
\begin{proof}
By the comparison principle we have
$$
\int_{\{u<v\}} H_m(u) \leq \int_{\{u<v\}} e^u d\mu \leq  \int_{\{u<v\}} e^v d\mu \leq \int_{\{u<v\}} H_m(v)     \leq \int_{\{u<v\}} H_m(u)  
$$
Thus all inequalities above become equalities and we infer that 
$H_m(u)(u<v)=\mu(u<v)=0$. Now, it suffices to apply the domination principle. 
\end{proof}

\begin{lem}
\label{lem: mixed equality}
Let $u,v,\f,\p_1,\p_2$ be bounded $\omega$-$m$-sh functions such that
$$
H_m(u) = f H_m(\f) + h_1 H_m(\p_1) \ , \ H_m(v) = g H_m(\f) + h_2 H_m(\p_2) ,
$$
where $f,g,h_1,h_2$ are non-negative bounded functions. Then for any $1\leq k\leq m-1$,
$$
\omega_u^{k} \wedge \omega_v^{m-k}\wedge \omega^{n-m} \geq f^{k/m}g^{(m-k)/m} H_m(\f). 
$$
\end{lem}
\begin{proof}

{\bf Step 1}: Assume that  $\f,\p_1,\p_2$ are smooth on $X$. 

\medskip

Let 
$f^j,g^j,h_1^j,h_2^j$ be uniformly bounded sequences of smooth non-negative functions which converge  in $L^1(X)$ to $f,g,h_1,h_2$ respectively. We can also assume that these sequences are normalized so that the corresponding measures have the same mass as $\int_X \omega^n$. We now use the main result of \cite{DK12} to solve the following equations
$$
H_m(u_j) = f_jH_m(\f) + h_1^j H_m(\p_1), \ H_m(v_j) =g_j H_m(\f) + h_2^j H_m(\p_2),
$$
where $u_j, v_j$ are smooth $\omega$-$m$-sh functions. 
We also normalize $u_j,v_j$ so that $\sup_X u_j=\sup_X u$ and 
$\sup_X v_j=\sup_X v$. It follows from Lemma \ref{lem: convergence} and Lemma \ref{lem: uniqueness bounded} that $u_j$ and $v_j$ converge in $\Capm$ to $u$ and $v$ respectively. Since $u_j$ and $v_j$ are smooth by G{\aa}rding's inequality (\cite{Ga59}) we get 
$$
\omega_{u_j}^{k} \wedge \omega_{v_j}^{m-k}\wedge \omega^{n-m} \geq (f_j)^{k/m}(g_j)^{(m-k)/m} H_m(\f).
$$
Now, the result follows by letting $j\to+\infty$. 

\medskip

\noindent{\bf Step 2}: Assume that  $f,g$ are quasi-continuous on $X$ and $\min(f,g)\geq \delta>0$ for some positive constant $\delta$. 
\medskip

It follows from \cite{LN14} that we can approximate $\f, \p_1,\p_2$ from above by decreasing sequences of smooth $\omega$-$m$-sh functions, say $(\f^j), (\p_1^j), (\p_2^j)$. 
Let $u_j\in \SHXo$ solve the equation
$$
H_m(u_j) = e^{u_j-u} \left[fH_m(\f^j) + h_1 H_m(\p_1^j)\right].
$$
Since $\f_j, \p_1^j$ are uniformly bounded and $u, f, h_1$ are bounded
(in fact $\delta\leq f\leq C$) we deduce from the comparison principle that $u_j$ is also uniformly bounded. Assume that $u_j$ converges in $L^1(X)$ to some bounded $\omega$-$m$-sh function $u_{\infty}$. It follows from Lemma \ref{lem: convergence} that 
$u_j$ converges in $\Capm$ to $u_{\infty}$. By letting $j\to +\infty$ we get
$$
H_m(u_{\infty}) =e^{u_{\infty}-u}H_m(u).
$$ 
Applying Lemma \ref{lem: comparison} we get $u=u_{\infty}$. Now, we do the same thing for $v$ to get a sequence $v_j$ converging in $\Capm$ to $v$. Applying Step 1 for  $u_j$ and $v_j$ we get
$$
\omega_{u_j}^{k} \wedge \omega_{v_j}^{m-k}\wedge \omega^{n-m} \geq 
e^{k(u_j-u)/m + (m-k)(v_j-v)/m}f^{k/m}g^{(m-k)/m} H_m(\f_j) .
$$
Since $u_j,v_j$ converge in $\Capm$ to $u$ and $v$ respectively and $f,g$ are quasi-continuous we can argue as in \cite{Kol05} (see Lemma \ref{lem: cap convergence}) to see that the right-hand side of the above inequality converges to 
$$
f^{k/m}g^{(m-k)/m} H_m(\f) ,
$$
while the left-hand side converges to $\omega_{u}^{k} \wedge \omega_{v}^{m-k}\wedge \omega^{n-m}$. 
\medskip

\noindent
{\bf Step 3}: Assume that $\min(f,g)\geq \delta>0$ for some positive constant $\delta$.

Let $f_j,g_j$ be uniformly bounded sequences of continuous non-negative functions which converge  in $L^1(X,H_m(\f))$ to $f,g$ respectively. 
We can assume also that $\min(f_j,g_j) \geq \delta$ for all $j$. Let $u_j$ solve the equation
$$
H_m(u_j) = e^{u_j-u} \left[f_jH_m(\f) + h_1 H_m(\p_1)\right].
$$
It follows from Lemma \ref{lem: comparison} that $u_j$ is uniformly bounded. We can argue as in Step 2 to see that $u_j$ converges in $\Capm$ to $u$. Now, do the same thing for $v$ and use Step 2 to get
$$
\omega_{u_j}^{k} \wedge \omega_{v_j}^{m-k}\wedge \omega^{n-m} \geq 
e^{k(u_j-u)/m + (m-k)(v_j-v)/m}f_j^{k/m}g_j^{(m-k)/m} H_m(\f) .
$$
The result follows by letting $j\to+\infty$.

\medskip

\noindent{\bf Step 4}: We now prove the general statement. 
Consider $f_{\vep}:=\max(f,\vep)$ and solve 
$$
H_m(u_{\vep}) = e^{u_{\vep}-u} \left[f_{\vep} H_m(\f) + h_1 H_m(\p_1)\right].
$$
Then $u_{\vep}\leq u$ and $u_{\vep}\geq \f/2 + \p_1/2 -C$ as follows from the comparison principle (see Lemma \ref{lem: comparison}). Do the same thing for $v$ and apply Step 3 to conclude.
\end{proof}

Now, we are ready to prove Theorem \ref{theorem 1}. For the sake of simplicity we only treat the case when there are only two functions instead of a collection of $m$ functions. The general case follows in an obvious way by changing the notation. 
\begin{thm}
\label{thm: mixed hessian inequality}
Let $u,v$ be functions in $\EcXo$ and $\mu$ be a positive Radon measure vanishing on $m$-polar sets such that 
$$
H_m(u) \geq f \mu \ , \ H_m(v) \geq g \mu ,
$$
where $f,g$ are non-negative integrable (with respect to $\mu$) functions. Then for any $1\leq k\leq m-1$, one has
$$
\omega_u^k \wedge \omega_v^{m-k}\wedge \omega^{n-m} \geq f^{k/m}g^{(m-k)/m} \mu .
$$
\end{thm}
\begin{proof}
We first treat the case when $u,v,f,g$ are bounded and $\mu=H_m(\f)$ for some bounded $\omega$-$m$-sh function $\f$. Fix $\delta>0$. For each $j>1$ let
$u_j$ solve the following equation 
$$
H_m(u_j) = e^{j(u_j-u)}\left[ ( 1-\delta ) f\mu + \delta  H_m(u)\right].
$$
One can solve the above equation by using the variational method exactly the same way as in \cite{LN14} (see Theorem \ref{thm: exponential}).
By the comparison principle one gets $u\leq u_j\leq u + j^{-1}\log( \delta^{-1})$. Thus $u_j$ converges uniformly on $X$ 
to $u$.
Now, do the same thing for $v$ and get a sequence $v_j$ which converges uniformly to $v$. Applying Lemma \ref{lem: mixed equality} and letting $j\to +\infty$ we get
$$
\omega_u^k \wedge \omega_v^{m-k}\wedge \omega^{n-m} \geq (1-\delta) f^{k/m}g^{(m-k)/m} \mu .
$$ 
Now, the result follows by letting $\delta \to 0$.
\medskip

If $f,g$ and $\f$ are bounded but $u,v$ are not bounded we can argue as follows. Consider $u_j=\max(u,-j)$ and $v_j=\max(v,-j)$. Then 
$$
H_m(u_j) \geq \ind_{\{u > -j\}} f H_m(\f) \ {\rm and }\ H_m(v_j) \geq \ind_{\{v > -j\}} f H_m(\f).
$$
We then apply the previous step for $u_j,v_j$ and let $j$ go to $+\infty$, noting that $H_m(\f)$ does not charge $m$-polar sets.
\medskip

For the general case observe that since $\mu$ does not charge $m$-polar sets, it follows from \cite{LN14} and the generalized 
Radon-Nikodym theorem \cite{Rai69} that we can write $\mu = h H_m(\f)$ for some 
$\f\in \SHXo \cap L^{\infty}(X)$ and $h\in L^1(H_m(\f))$. We thus can assume that $\mu =H_m(\phi)$ for some bounded $\omega$-$m$-sh function 
$\phi$. Now, consider $f_j:=\min(f,j)$ and $g_j:=\min(g,j)$. Applying the first step and letting $j\to+\infty$ we get the result.  
\end{proof}
\begin{rem}
In the first step of the proof, instead of solving a family of equations with parameter $j$ we can argue as follows. Observe first that  $H_m(u)-f\mu$ is a positive Radon measure dominated by $H_m(u)$ with $u\in L^{\infty}(X)$. Then one can   find  $\p\in \SHXo \cap L^{\infty}(X)$ such that
$$
H_m(\p)=e^{\p} [H_m(u)-f\mu] .
$$
Indeed, the existence of a solution follows from the variational approach and it is bounded since there is a bounded subsolution (see also \cite{Zer14} for a more detailed discussion about the envelop method for the complex Monge-Amp\`ere equation).
\end{rem}

We finish this section by presenting how {\it local} inequalities for Hessian measures stem from their global counterpart. More precisely we prove the following theorem:
\begin{thm}
 \label{thm: global to local}
Let $\omega$ be a germ of a K\"ahler metric near $0\in \C^n$. Let also $u_1, u_2,\cdots, u_m$ be bounded $m$-subharmonic functions satisfying $(dd^cu_j)^m\wedge\omega^{n-m}\geq f_j\mu$ for some non negative Borel measure vanishing on all $m$-polar sets. Then
\begin{equation}\label{eq: mixed local}
dd^cu_1\wedge\cdots\wedge dd^cu_m\wedge\omega^{n-m}\geq(f_1\cdots f_m)^{1/m}\mu.
\end{equation}
\end{thm}
Before we start the proof we recall the following standard fact, based on the patching of local potentials (see, for example \cite{Ve10}, Lemma 3.8 for a discussion):
\begin{lem} Let $M$ be the unit ball in $\C^n$ equipped with the K\"ahler metric $\omega=dd^c\phi$, with $\phi$ bounded. Fix a smaller ball $B$ centered at $0$. Then $(M,\omega)$ admits an isometric embedding into a compact complex torus $(X,\tilde{\omega})$ such that $\omega|_{B}=\tilde{\omega}_{Im(B)}$ with $Im(B)$ being the image of $B$ under the isometry. Moreover $\tilde{\omega}$ can be taken to be the standard flat metric on a neighborhood of $X\setminus Im(M)$. 
\end{lem}
Now we can prove the local mixed Hessian inequalities:
\begin{proof}[Proof of theorem \ref{thm: global to local}]
 Suppose first that all the functions $u_1,\cdots u_n$ are bounded near $0\in C^n$. It is enough to establish the inequality
 in a small neighborhood of $0$. Shrinking that neighborhood if necessary we may assume that $\omega=dd^c\phi$ for some
 bounded plurisubharmonic function $\phi$.

 Exploiting the previous lemma we may assume that (with the obvious identifications) $\omega$ is a K\"aler form on a compact 
K\"ahler manifold $X$. Shrinking the domain further if necessary, we can patch each of the functions $u_j-\phi$ with a suitable global $\omega$-$m$-subharmonic function with an isolated pole at $0\in X$ (strictly speaking we have to use the regularized maximum technique instead of ordinary maximum- see \cite{Ve10} for a discussion) we get {\it global} $\omega$-$m$-subharmonic functions $\tilde{u}_j$ such that $\tilde{u}_j=u_j-\phi$ in a neighborhood of $0\in X$. Now by the global inequality for mixed Hessian measures we get
$$\omega_{\tilde{u}_1}\wedge\cdots\wedge\omega_{\tilde{u}_m}\wedge\omega^{n-m}\geq (f_1\cdots f_m)^{1/m}\mu$$
in a neighborhood of $0\in X$ which is exactly what we seek.

Now the passage from $u_j$ bounded to general $u_j$ is done exactly like in the global argument. 
\end{proof}

\section{uniqueness}\label{Uniq}
The main result of this section is the following uniqueness result for the normalized solutions to the complex hessian equations:
\begin{thm}\label{thm: uniqueness}
 $u,v$ be functions in $\EcXo$ and $\mu$ be a positive Radon measure vanishing on $m$-polar sets such that 
$$H_m(u)=\mu=H_m(v).$$
Then $u-v$ is a constant.
\end{thm}
\begin{proof}
Given the inequality for mixed hessian measures the argument pretty much follows the one from \cite{Dinew09}. We present the details for the sake of completeness.

Suppose on contrary that $u-v\neq$ const. Note that $\forall t\in \mathbb R\cup\{\infty\}\cup\{-\infty\}$ the sets $\{u-v=t\}$ are Borel and hence at most countably many of these are charged by $\mu$. Observe that by assumption the sets with $t=+\infty$ or $-\infty$ are massless for they are $m$-polar . Just like in \cite{Dinew09} we prove that actually precisely one of the remaining sets has positive $\mu$-mass. For if not then there must exist $t\in\R$ such that $\{u-v=t\}$ is massless, while $\mu(\{u-v<t\}), \mu(\{u-v>t\}) >c>0$ for some constant $c\leq 1/2$. Note that after adding a constant we can and will assume $t$ to be zero.

Consider the new measure 
$$\widehat{\mu}:=\begin{cases} C\mu, & on\ \{u<v\}\\
  0, & on\ \{u\geq v\},              \end{cases}$$
where $C>1$\ is a nonnegative normalization constant so that $\widehat{\mu}$\ is a nonnegative probability measure (note that this is possible, since, by assumption, $\mu$\ charges the set $\{u\geq v\}$).

Of course $\widehat{\mu}$\ does not charge pluripolar sets either (and is also a Borel measure since the set $\{u\geq v\}$\ is Borel). By \cite{LN14} we can solve the Hessian equation
$$H_m(w)=\widehat{\mu},\ \ w\in\EcXo,\ sup_X w=0.$$

Consider the set inclusion
$$U_t:=\{(1-t)u<(1-t)v+tw\}\subset\{u<v\}$$
for every $t\in(0,1)$. Hence on $U_t$ we have
$$
\omega_{(1-t)v+tw}\wedge\omega_u^{m-1}\wedge\omega^{n-m}\geq(1-t)\mu+tC^{1/m}\mu=(1+t(C^{1/m}-1))H_m(u),
$$
where we have made use of Theorem \ref{thm: mixed hessian inequality}.

Exploiting the partial comparison principle (see \cite{LN14}) we get
\begin{align*}
&(1+((C)^{1/m}-1)t)\int_{U_t}H_m(u)\leq\int_{U_t}\omega_{(1-t)v+tw}\wedge\omega_u^{m-1}\wedge\omega^{n-m}\\
&\leq\int_{U_t}\omega_{(1-t)u+t0}\wedge\omega_u^{m-1}\wedge\omega^{n-m}=(1-t)\int_{U_t}
H_m(u)+t\int_{U_t}\omega_u^{m-1}\wedge\omega^{n-m+1}.
\end{align*}
 In other words we get
\begin{equation}\label{1}
C^{1/m}\int_{U_t}\mu\leq\int_{U_t}\omega_u^{m-1}\wedge\omega^{n-m+1}.
\end{equation}
Note that the inequality for Hessian measures coupled with total volume considerations yields
$$\omega_u^k\wedge\omega_v^{m-k}\wedge\omega^{n-m}=\mu$$
for any $k\in\{0,1,\cdots,m\}$, see Corollary 2.2. in \cite{Dinew09}. 
But then the same argument with $u$ exchanged by $v$ leads to the inequality
$$C^{1/m}\int_{U_t}\mu\leq\int_{U_t}\omega_v^{m-1}\wedge\omega^{n-m+1}.$$ 

If we let now $t$ to zero we observe that $U_t$ converges to $\{u<v\}\setminus\{w=-\infty\}$, but $\{w=-\infty\}$ is an $m$-polar set and hence is negligible with respect to $\mu$. Thus we get 
$$C^{1/m}\int_{\{u<v\}}\mu\leq\int_{\{u<v\}}\omega_v^{m-1}\wedge\omega^{n-m+1},$$
as well as
$$C^{1/m}\int_{\{u<v\}}\mu\leq\int_{\{u<v\}}\omega_u^{m-1}\wedge\omega^{n-m+1}.$$ 

Playing the same game on $\{u>v\}$ (namely we define a measure like $\hat{\mu}$ with respect to the set $\{u>v\}$) we also get for a different constant $D>1$ the inequality

$$D^{1/m}\int_{\{u>v\}}\mu\leq\int_{\{u>v\}}\omega_v^{m-1}\wedge\omega^{n-m+1}.$$

Coupling these with the fact that $\{u=v\}$ is massless by construction we end up with the inequalities  
$$\min\{C^{1/m},D^{1/m}\}\int_X\mu\leq \int_{\{u<v\}}\omega_v^{m-1}\wedge\omega^{n-m+1}+\int_{\{u>v\}}\omega_v^{m-1}\wedge\omega^{n-m+1}\leq\int_X\mu,$$
so $1<\min\{C^{1/m},D^{1/m}\}\leq1$, a contradiction.

Thus we know that $\mu(\{u\neq v\})=0$. Just like in \cite{Dinew09} our next and final task will be to prove analogous mass vanishing for $H_j(u), H_j(v),\ j=0,1,\cdots,m-1$. Indeed, consider the sets
$$V_{t,j}:=\{u+(t/j)u_j+(3/2)t
<v\}\subset\{u<v\},$$
where $u_j:=\max\{u,-j\}$.

The partial comparison principle results in
$$\int_{V_{t,j}}\omega_u^{k}\wedge\omega_v^{m-1-k}\wedge(\omega_v+(t/j)\omega)\wedge\omega^{n-m}\leq
\int_{V_{t,j}}\omega_u^{k}\wedge\omega_v^{m-1-k}\wedge(\omega_u+(t/j)\omega_{u_j})\wedge\omega^{n-m}.$$
Now, (recall $\omega_u^{k}\wedge\omega_v^{m-k}\wedge\omega^{n-m}=\mu,\ \forall k\in\{0,\cdots m\}$) the equation above reduces to
$$\int_{V_{t,j}}\omega_u^{k}\wedge\omega_v^{m-1-k}\wedge\omega^{n-m+1}\leq
\int_{V_{t,j}}\omega_u^{k}\wedge\omega_v^{m-1-k}\wedge\omega_{u_j}\wedge\omega^{n-m}.$$
Note that $V_{t,j}$\ is a decreasing sequence of sets in terms of $j$. Letting $j\rightarrow\infty$, using vanishing on pluripolar sets and then letting $t$ to zero we obtain
$$\int_{\{u<v\}}\omega_u^{k}\wedge\omega_v^{m-1-k}\wedge\omega^{n-m+1}\leq
\int_{\{u<v\}}H_m(u)=0.$$
In the same vein the measures $\omega_u^{k}\wedge\omega_v^{m-1-k}\wedge\omega^{n-m+1}$ put no mass on $\{u>v\}$. Finally exchanging $\omega_u^{k}\wedge\omega_v^{m-1-k}\wedge\omega^{n-m}$\ in the argument above with  $\omega_u^k\wedge\omega_v^{m-2-k}\wedge\omega^{n-m+1}$\ we obtain mass vanishing on $\{u<v\}$ for the measures
$\omega_u^k\wedge\omega_v^{m-2-k}\wedge\omega^{n-m+2},\ k=0,1,\cdots, m-2$. An easy induction finally yields that $\omega^n$ has its mass supported on $\{u=v\}$ which is impossible unless $u$ equals $v$.
\end{proof}

\section{An example}\label{example}
In this section we give an example which shows that vanishing on $m$-polar sets cannot be removed from the assumptions. The example (adapted from \cite{Dinew07}) actually shows that our mixed Hessian inequality fails outside the class $\EcXo$. 
\begin{exa}
Let $n\geq 3, m=2$ and $\omega=dd^c |z|^2$ be the flat K\"ahler metric in 
$\C^n$. Consider the following functions
$$
u_k(z):= \max \left(\frac{1}{k}\log |z_1|, k^2 \log |z_2|\right), \
{\rm and}\ v_k(z):=\max \left(\frac{1}{k}\log |z_2|, k^2 \log |z_1|\right). 
$$
Then 
$$
(dd^c u_k)^2 \wedge \omega^{n-2} (dd^c v_k)^2 \wedge \omega^{n-2} = \frac{(2\pi)^2 k }{2} [z_1=z_2=0],
$$
where $[Z]$ means the current of integration along $Z$.  But 
$$
dd^c u_k \wedge dd^c v_k\wedge \omega^{n-2} = \frac{(2\pi)^2  }{2k^2} [z_1=z_2=0],
$$
which violates (\ref{eq: mixed local}) when $k>1$. 
\end{exa}
Note that in this example $u_k, v_k$ (which are plurisubharmonic functions) belong to the domain of definition of $H_m$ since they belong to $W^{1,2}(\C^n)$ (see Lemma \ref{lem: W12 Blocki}). But their Hessian measures charge the $m$-polar set $\{z_1=z_2=0\}$. 
\begin{proof}
From \cite[Example 4.1]{Dinew07} we know that $(dd^c u_k)^2=(dd^c v_k)^2$, when considered as measure in $\C^2$, is the Dirac measure of the origin with coefficient $a(k)= \frac{(2\pi)^2 k }{2}$. Thus the Hessian measure of $u_k$ and $v_k$ is the  integration along the line $\{z_1=z_2=0\}$ with coefficient $a(k)$,
 while $dd^c u_k \wedge dd^c v_k \wedge \omega^{n-2}$
is the  integration along the same line with coefficient 
$b(k)= \frac{(2\pi)^2  }{2k^2}$. When $k>1$ this violates inequality (\ref{eq: mixed local}). 
\end{proof}
\begin{rem}
 The example above has Hessian measure charging some non-dicrete analytic $m$-polar set. It is interesting to note that in the
 plurisubharmonic case it is a deep open problem whether such a function actually exists.
\end{rem}

\end{document}